\documentclass[12pt]{article}
\usepackage{amsmath, amsfonts, amsthm, amssymb, verbatim}
\usepackage{graphicx}
\usepackage[mathcal]{euscript}
\usepackage{amstext}
\usepackage{amssymb,mathrsfs}
\usepackage{bbm}
\usepackage[latin1]{inputenc}
\newcommand{\R}{\mathbb R}

  \newcommand{\E}{\mathbb E}

\newcommand{\PP}{\mathbb P}

\newcommand{\calF} {\ensuremath {\mathcal{F}}}
\newcommand{\dive}{{\rm div}}

\newtheorem{theorem}{Theorem}[section]

 \newtheorem{remark}[theorem]{Remark}
\newtheorem{lemma}[theorem]{Lemma}

\newtheorem{definition}[theorem]{Definition}
\newtheorem{hypothesis}[theorem]{Hypothesis}

\begin{document}

\title{  Regularization by noise in   (2x 2) hyperbolic systems of conservation law. }

\author{
Christian Olivera\footnote{Departamento de Matem\'{a}tica, Universidade Estadual de Campinas, Brazil.
E-mail:  {\sl  colivera@ime.unicamp.br}.
}}

\date{}

\maketitle


\vspace{0.3cm} \noindent {\bf MSC2010 subject classification:} 60H15, 
 35R60, 
 35L02, 
 60H30, 35L40. 


%
\begin{abstract}
 In this paper we study a non strictly systems of conservation law
by stochastic perturbation. We  show the existence and uniqueness of the solution. We do  not assume that $BV$-regularity 
for the initial conditions. The proofs  are based on the concept of entropy solution and in the characteristics method (in the influence of noise). This is the first result on the regularization by noise in hyperbolic systems of conservation law. 
 \end{abstract}
%
\maketitle

%

\section {Introduction} \label{Intro}

A large number of problems in physics and engineering are modeled by systems of conservation
laws

\begin{equation}\label{trasports}
    \partial_t u(t, x) + div ( f(u(t,x)) ) = 0 \, ,
\end{equation}

\noindent  here $u=u(t,x)$ is called the conserved quantity, while $F$ is the flux. Examples for hyperbolic systems of conservation laws include
the shallow water equations of oceanography, the Euler equations of gas dynamics, the
magnetohydrodynamics (MHD) equations of plasma physics, the equations of nonlinear
elastodynamics and the Einstein equations of general relativity. 
When smooth initial data are considered, it is well known that the solution
can develop shocks within finite time. Therefore, global solutions can only
be constructed within a space of discontinuous functions. Moreover,  when discontinuities are present, weak solutions may not be unique. 
 A central issue is to regain uniqueness by imposing appropriate selection criteria. 
The  well-posedness theorems within the class of entropy solutions, for the scalar case,  were
established by Kruzkov(see \cite{Krus}). It is well known that the main techniques of abstract functional analysis do not apply to hyperbolic systems. Solutions cannot be represented as fixed points of continuous transformations, or in
variational form, as critical points of suitable functionals.  For the above reasons, the theory of hyperbolic conservation laws has largely developed by ad hoc methods. We refer to \cite{Bress}, \cite{Bian}
\cite{Dafermos} and \cite{Serre}. The well-posedness general system of conservation laws has been established only for initial data with sufficiently small total variation, see  \cite{Bress} and the
references therein.

 We consider the following systems of conservation law

\begin{equation}\label{deter}
 \left \{
\begin{aligned}
    &\partial_t v(t, x) + Div  \big( F(v) \big )= 0 \,  \\
    \\[5pt]
		&\partial_t u(t, x) +  Div (v u )  = 0 \, .
		\end{aligned}
\right .
\end{equation}

We point that in  the $L^{1}\cap L^{\infty}$ setting this systems ill-posedness since  the classical  DiPerna-Lions-Ambrossio theory of uniqueness   of distributional solutions  for transport/ continuity equation  does not apply when the drift has   $L^{1}\cap L^{2}$ regularity, see  \cite{ambrisio} and  \cite{DL}.  Also see   \cite{ambrisio2}  and  \cite{lellis}  for new developments in the theory.
We point under strong assumption on the coefficients and  initial conditions P. Le Floch  in  \cite{Lef2} solved this problem using Volpert multiplication of distributions.
\medskip
In contrast with its deterministic counterpart, the singular stochastic continuity/transport equation with multiplicative noise is well-posed.  
The addition of a stochastic noise is often used to account for
numerical, empirical or physical uncertainties. In \cite{AttFl11,Beck,Fre1,Fre2,FGP2,MNP14,Moli}, well-posedness and regularization by linear multiplicative noise for continuity/transport equations have been obtained. We refer to \cite{Moli} for more details on the literature.

 In this paper we study the influence of the noise in the hyperbolic systems  (\ref{deter}). More precisely,  we consider  following  stochastic  systems of conservation law

\begin{equation}\label{trasport}
 \left \{
\begin{aligned}
    &\partial_t v(t, x) + Div  \big( ( F(v(t,x))  \big )= 0 \, ,
    \\[5pt]
		    &\partial_t u(t, x) + Div  \big( (v(t,x) + \frac{d B_{t}}{dt}) \cdot  u(t, x)  \big )= 0 \, ,
    \\[5pt]
    &v|_{t=0}=  v_{0}, \  u_{t=0}=u_{0}\  \, .
\end{aligned}
\right .
\end{equation}
Here, $(t,x) \in [0,T] \times \R$, $\omega \in \Omega$ is an element of the probability space $(\Omega, \PP, \calF)$ and $B_{t}$ is a standard Brownian motion in $\mathbb{R}$. The stochastic integration is to be understood in the Stratonovich sense.
The Stratonovich form is the natural one for several reasons , including physical intuition related to the Wong-Zakai principle.  

\medskip

The main issue of this paper is to prove  existence and uniqueness of  entropy-weak solutions for  the  stochastic systems
of the conservation law (\ref{trasport}).  We do not assume  $BV$- regularity for the initial conditions.
We use  the   entropy formulation of conservation law and we  employ the  stochastic characteristics in order to obtain  a unique  solution to the  one-dimensional stochastic equation with a bounded measurable drift coefficient.
 We adapted  the ideas in \cite{Moli} and \cite{Ol} in our context where the drift term in the continuity equation depend on time and it is bounded and integrable. 

Throughout of this paper, we fix a stochastic basis with a $d$-dimensional Brownian motion $\big( \Omega, \mathcal{F}, \{\mathcal{F}_t: t \in [0,T] \}, \mathbb{P}, (B_{t}) \big)$.

\subsection{One Example}
 We consider the systems

\begin{equation}\label{deter3}
 \left \{
\begin{aligned}
    &\partial_t v(t, x) + Div  \big( \frac{1}{2} v^{2}(t,x)  \big )= 0 \, ,
    \\[5pt]
		    &\partial_t u(t, x) + Div  \big( ( v(t,x)   u(t, x)  \big )= 0 \, ,
    \\[5pt]
    &v|_{t=0}=  v_{0}, \  u_{t=0}=u_{0}\  \, ,
\end{aligned}
\right .
\end{equation}

here $v$ is the velocity and $u$  the density of the particles.
This system has applications in cosmology, the  model
describes the evolution of matter in the last stage of the expansion of the
universe as cold dust moving under gravity alone and the laws are governed
by the system (\ref{deter3}). Clearly the eigenvalues  are equal  $\lambda_1 = \lambda_2 = v$. Thus the  system (\ref{deter3}) 
is not  strictly hyperbolic.  The first equation of  (\ref{deter3})-Burgers equation is
known to  develop  singularities in finite time even if the initial data 
$v_{0}$ is smooth, and it is not at all obvious to solve the second equation. One question that remains is a well-posedness theory and large time behaviour of solution. In \cite{Jose} and \cite{Tan} the authors  proved existence of weak solutions via $\delta-$ shock for Riemann initial condition.  Another approach of nonconservative product can be found in the 
work of J.F. Colombeau \cite{Colom} .

\subsection{Scalar case.}
We point that recently there has been an interest in studying the effect of stochastic forcing on
nonlinear conservation laws driven by space-time white noise, see  \cite{chen2,Debu,Feng,Hof}. For other hand, in \cite{lions} and \cite{lions2}  the authors  introduced the theory of  pathwise  solutions to study the stochastic conservation law  
driven by  continuous noise.

\subsection{Possible extensions.} We point that our approach can be apply to other class of non-coupled systems like

\begin{equation}
 \left \{
\begin{aligned}
    &\partial_t v(t, x) =P(v)  \,  \\
    \\[5pt]
		&\partial_t u(t, x) +  Div (f(v) u )  = 0 \, .
		\end{aligned}
\right .
\end{equation}

where the $P$ is some differential operator.  For instance  for the  systems 
of the Hamilton-Jacobi and the continuity equations, see \cite{ST}. 
The problem of coupled systems is much more complicated and we shall consider it in future investigations.

\subsection{Hypothesis}

We assume the following conditions 
\begin{hypothesis}\label{hyp1}
The  flux  $F$ satisfies
\begin{equation}\label{cond3-1}
 F\in C^{1} 
\end{equation}
and the initial condition holds
\begin{equation}\label{weight}
  v_{0}\in L^{\infty}(\R)\cap L^{1}(\R), u_0 \in L^{2}(\R)\cap L^{1}(\R). 
\end{equation}
\end{hypothesis}


\section{Existence}

\subsection{Definition of solutions}

\begin{definition} Let $\eta\in C^{1}(\R)$ be a convex function. If there exist $q\in C^{1}(\R)$ such that  for all $v$
\[
\eta^{\prime}(v) F^{\prime}(v)= q^{\prime}(v)
\] 
then $\eta,q$ is called  an entropy-entropy flux pair of the conservation law 

\[
\partial_t v(t, x) + Div  \big( f(v) \big )=0, \ v(t,0)=v_{0}(x).
\]

\end{definition}

\begin{definition}\label{defisoluH}
The  stochastic process $v \in L^{\infty}([0,T]\times \R)\cap L^{\infty}([0,T],L^{1}(\R))$ and 
$u\in L^\infty([0,T], L^{2}(\Omega \times \R )) \cap L^1([0,T] \times \Omega \times \R ) $ are called a  entropy  weak solution of  the stochastic hyperbolic systems  \eqref{trasport} when:

\begin{itemize}

\item $v$ is entropy solution of the conservation law 

\[
\partial_t v(t, x) + Div  \big( F(v) \big )=0, \ v(t,0)=v_{0}(x).
\]

That is, if for every entropy flux pair $\eta, q$  we have

\[
\partial_t \eta(v) + Div (q(v))\leq 0
\]

 in the sense of distribution. 
\item For any $\varphi \in C_0^{\infty}(\R)$, the real valued process $\int  u(t,x)\varphi(x)  dx$ has a continuous modification which is an $\mathcal{F}_{t}$-semimartingale, and for all $t \in [0,T]$, we have $\mathbb{P}$-almost surely
\begin{equation} \label{DISTINTSTR}
\begin{aligned}
    \int_{\R} u(t,x) \varphi(x) dx = &\int_{\R} u_{0}(x) \varphi(x) \ dx
	+ \int_{0}^{t} \!\! \int_{\R}   u(s,x)   \, v(t,x) \partial_x \varphi(x) \ dx ds
\\[5pt]
    & + \int_{0}^{t} \!\! \int_{\R}   u(s,x) \ \partial_x \varphi(x) \ dx \, {\circ}{dB_s} \, .
\end{aligned}
\end{equation}
\end{itemize}
\end{definition}

\begin{remark}\label{lemmaito}
Using the same idea as in Lemma 13 \cite{FGP2}, one can write the problem (\ref{DISTINTSTR}) in It\^o form as follows, a  stochastic process $u\in   L^\infty([0,T], L^{2}(\Omega \times \R ) ) \cap L^1([0,T] \times \Omega \times \R ) $  is solution  of the SPDE (\ref{DISTINTSTR}) iff for every test function $\varphi \in C_{0}^{\infty}(\mathbb{R})$, the process $\int u(t, x)\varphi(x) dx$ has a continuous modification which is a $\mathcal{F}_{t}$-semimartingale and satisfies the following It\^o's formulation

\[
\begin{aligned}
    \int_{\R} u(t,x) \varphi(x) dx = &\int_{\R} u_{0}(x) \varphi(x) \ dx
	+ \int_{0}^{t} \!\! \int_{\R}   u(s,x)   \, v(t,x) \partial_x \varphi(x) \ dx ds
\\[5pt]
    & + \int_{0}^{t} \!\! \int_{\R}   u(s,x) \ \partial_x \varphi(x) \ dx \, dB_s \,  + \frac{1}{2} \int_{0}^{t} \!\! \int_{\R}   u(s,x) \ \partial_x^{2} \varphi(x) \ dx \, ds.
\end{aligned}
\]

\end{remark}

\subsection{Existence.}

We shall  prove existence of  solutions  under hypothesis \ref{hyp1}.
\begin{lemma}
Assume that hypothesis \ref{hyp1} holds. Then there exists entropy-weak solution of the hyperbolic systems  \eqref{trasport}.
\end{lemma}

\begin{proof}

{\it Step 1: Conservation law .} According to the classical theory of conservation law, see for instance \cite{Dafermos}, 
 we have that there exists a unique entropy solution of the conservation law

\[
\partial_t v(t, x) + Div  \big( F(v) \big )=0, \ v(t,0)=v_{0}(x).
\]

If the the initial condition $v_{0}\in L^{1}(\R)\cap L^{\infty}(\R)$ then the solution
 $v \in L^{\infty}([0,T]\times \R)\cap L^{\infty}([0,T],L^{1}(\R))$. 
\bigskip

{\it Step 2: Primitive of $v$.} It easy to see that for any test function 
$\varphi \in C_{0}^{\infty}(\mathbb{R})$  we have

\[
\begin{aligned}
    \int_{\R} v(t,x) \varphi(x) dx = &\int_{\R} v_{0}(x) \varphi(x) \ dx
	+ \int_{0}^{t} \!\! \int_{\R}   F(v(s,x))  \partial_x \varphi(x) \ dx ds.
\end{aligned}
\]

since  any entropy solution is also a weak solution. 

Let $\{\rho_\varepsilon\}_\varepsilon$ be a family of standard symmetric mollifiers. Then we obtain

\[
\begin{aligned}
    \int_{\R} v(t,y) \rho_\varepsilon(x-y) dy = &\int_{\R} v_{0}(y) \rho_\varepsilon(x-y) dy 
	+ \int_{0}^{t} \!\! \int_{\R}   F(v(s,y))  \partial_y \rho_\varepsilon(x-y) dy ds.
\end{aligned}
\]

Integrating  we get

\[
\begin{aligned}
    \int_{0}^{z} v_\varepsilon(t,x) dx = &\int_{0}^{z} v_{0}^\varepsilon(x) dz 
	+ \int_{0}^{t}    (F(v)\ast \rho_\varepsilon)(z)  ds.
\end{aligned}
\]

We denoted $\int_{0}^{z} v_\varepsilon(t,x) dx=\bar{v}_\varepsilon(t,x)$.

\bigskip

{\it Step 3: Regularization.} 
We  define the family of  regularized coefficients given by 
$$
v^{\epsilon}(t,.) = (v(t,x) \ast_{x} \rho_\varepsilon)(t,.).
$$

Clearly we observe that, for every  $\varepsilon>0$, any element $v^{\varepsilon}$, $u_0^\varepsilon$ are smooth (in space) and  with bounded derivatives of all orders.  We observe that to study the stochastic continuity equation (SCE) (\ref{DISTINTSTR}) is equivalent to study the stochastic transport equation given by (regularized version):
\begin{equation}\label{STE-reg}
 \left \{
\begin{aligned}
    &d u^\varepsilon (t, x) +  \nabla u^\varepsilon (t, x)  \cdot \big( v^\varepsilon (t,x)  dt +
 \circ d B_{t} \big) +\dive b^{\varepsilon}(x) \,u^\varepsilon (t,x) dt = 0\, ,
    \\[5pt]
    &u^\varepsilon \big|_{t=0}=  u_{0}^\varepsilon
\end{aligned}
\right .
\end{equation}
Following the classical theory of H. Kunita \cite[Theorem 6.1.9]{Ku} we obtain that

\[
u^{\varepsilon}(t,x) =  u_{0}^{\varepsilon} (\psi_t^{\varepsilon}(t,x))  J\psi_t^{\varepsilon}(t,x),
\]
is the unique solution to the regularized equation \eqref{STE-reg}, where $\phi_t^{\varepsilon}$ is the flow associated  to the following stochastic differential equation (SDE):
\begin{equation*}
d X_t = v^\varepsilon (t,X_t) \, dt + d B_t \, ,  \hspace{1cm}   X_0 = x \,,
\end{equation*}
and $\psi_t^{\varepsilon}$ is the inverse of $\phi_t^{\varepsilon}$.

\bigskip

{\it Step 4: It\^o Formula .} Applying the It\^o formula to  $\bar{v}_\varepsilon(t,X_t^{\epsilon})$ we deduce 

\[
    \bar{v}_\varepsilon(t,X_t^{\epsilon})=   \int_{0}^{X_t^{\epsilon}} u_{0}^{\epsilon}(x) dx 
	+ \int_{0}^{t}    (F(v)\ast \rho_\varepsilon)(s,X_s^{\epsilon})  ds +  \int_{0}^{t}    v_\varepsilon^{2}(s,X_s^{\epsilon}) ds 
\]

\[
   +  \int_{0}^{t}    v_\varepsilon(s,X_s^{\epsilon}) dB_{s}
	+ \frac{1}{2} \int_{0}^{t}  (\partial_{x}v_\varepsilon\big)(s,X_s^{\epsilon}) ds 
\]

{\it Step 5:  Boundeness.}  We observe that

\[
\|\bar{v}_\varepsilon(t,X_t^{\epsilon})\|_{L^{\infty}(\Omega\times [0,T]\times \R)}
\leq \| v\|_{L^{\infty}( [0,T],L^1(\R))},
\]

\[
\|\int_{0}^{X_t^{\epsilon}} v_{0}^{\epsilon}(x) dx\|_{L^{\infty}(\Omega\times [0,T]\times \R)}
\leq \| v_{0}\|_{L^{1}( \R)},
\]

\[
\|\int_{0}^{t}    (F(v)\ast \rho_\varepsilon)(s,X_s^{\epsilon})  ds\|_{L^{\infty}(\Omega\times [0,T]\times \R)}
  \leq  C \| F(v)\|_{L^{\infty}},
\]

\[
\| \int_{0}^{t}    v_\varepsilon^{2}(s,X_s^{\epsilon}) ds\|_{L^{\infty}(\Omega\times [0,T]\times \R)}
  \leq  C \| v\|_{L^{2}([0,T], L^{\infty}(\R))}^{2}.
\]

{\it Step 6 : Estimation on Jacobain.}

\noindent We denote

\begin{align*}
\mathcal{E}\bigg(\int_0^t v_{\epsilon}(s,X_s)dB_s\bigg)=
\exp\bigg\{\int_0^t v_{\epsilon}(s,X_s^{\epsilon})dB_s-\frac{1}{2} \int_0^{t} v_{\epsilon}^{2}(s,X_s^{\epsilon})ds \bigg\},
\end{align*}

We note that $\partial_x X_{t}$ satisfies

\[
\partial_x X_{t}=\exp\bigg\{ \int_{0}^{t} (\partial_x v_{\epsilon})(s,X_{s}) \  ds  \bigg\}. 
\]

From steps 4-5 we have

\[
\E | \partial_x X_{t}|^{-1}\leq C \E \mathcal{E}\bigg(\int_0^t v_{\epsilon}(s,X_s)dB_s\bigg).
\]

We observe  that the  processes  $ \mathcal{E}\bigg(\int_0^t v_{\epsilon}(s,X_s)dB_s\bigg)$,
   is martingale with expectation equal to one.  The we conclude that

	\[
\E | \partial_x X_{t}|^{-1}\leq C .
\]

{\it Step 7: Passing to the limit .}

Making the change of variables $y=\psi_t^{\varepsilon}(x)$ we have that

\begin{align*}
\int_{\R} \E[|u^{\varepsilon}(t,x)|^2]\,  dx & =   \int_{\R} |u_{0}^{\varepsilon} (y)|^2  \E| J\phi_t^{\varepsilon}|^{-1}    dy.
\end{align*}

From  step  6 we have 

\begin{equation}\label{eq0}
\int_{\R} \E[|u^{\varepsilon}(t,x)|^2]\,  dx \leq C .
\end{equation}

Therefore, the sequence $\{u^{\varepsilon}\}_{\varepsilon>0}$ is bounded in $L^\infty([0,T], L^{2}(\Omega \times \R ) ) \cap L^1([0,T] \times \Omega \times \R )$. Then  there exists a convergent subsequence, which we denote also by $u^{\varepsilon}$, such that converge weakly in $L^\infty([0,T], L^{2}(\Omega \times \R ) )$ to some process $u\in L^\infty([0,T], L^{2}(\Omega \times \R )) \cap L^1([0,T] \times \Omega \times \R )  $.

Now, if $u^{\varepsilon}$ is a solution of \eqref{STE-reg}, it is also a weak solution, that is, for any test function $\varphi\in C_0^{\infty}(\R)$, $u^{\varepsilon}$ satisfies (written in the Itô form):
\begin{align*}
\int_{\R} u^{\varepsilon}(t,x) \varphi(x) dx = &\int_{\R} u^{\varepsilon}_{0}(x) \varphi(x) \ dx + \int_{0}^{t} \!\! \int_{\R}   u^{\varepsilon}(s,x)   \,  v^{\varepsilon}(s,x) \partial_x \varphi(x) \ dx ds \\
    & + \int_{0}^{t} \!\! \int_{\R}   u^{\varepsilon}(s,x) \ \partial_x \varphi(x) \ dx \, dB_s \,  + \frac{1}{2} \int_{0}^{t} \!\! \int_{\R}   u^{\varepsilon}(s,x) \ \partial_x^{2} \varphi(x) \ dx \, ds\,.
\end{align*}
Thus, for prove existence of the SCE \eqref{trasport} is enough to pass to the limit in the above equation along the convergent subsequence found. This is made through of the same arguments of \cite[theorem 15]{FGP2}.
\medskip

\end{proof}


\section{Uniqueness.}

In this section, we shall present a uniqueness theorem
for the SPDE (\ref{trasport})
\begin{theorem}\label{uni2}
Under the conditions of hypothesis \ref{hyp1}, uniqueness holds for  entropy -weak solutions of the hyperbolic  problem \eqref{trasport}.
\end{theorem}

\begin{proof}

{\it Step 1: Set of solutions.}  The uniqueness of the conservation law

\[
\partial_t v(t, x) + Div  \big( F(v) \big )=0, \ v(t,0)=v_{0}(x).
\]

follows from  the classical theory of entropy solutions. 

\bigskip

{\it Step 2: } We remark that the set of   solutions of  equation (\ref{DISTINTSTR}) 
is a linear subspace of $ L^{\infty}([0,T]\times R, L^{2}(\Omega)) \cap L^1([0,T] \times \Omega \times \R )$, because the stochastic continuity equation is linear, and the integrability conditions is a linear constraint. Therefore, it is enough to show that a  $u$ with initial condition $u_0= 0$ vanishes identically.

\bigskip

{\it Step 1:  Primitive of the solution.} We define $V(t,x)=\int_{-\infty}^{x} u(t,y) \ dy $. We consider a  nonnegative smooth cut-off function $\eta$ supported on the ball of radius 2 and such that  $\eta=1$ on the ball of radius 1. For any $R>0$, we introduce the rescaled functions $\eta_R (\cdot) =  \eta(\frac{.}{R})$.
Let be $\varphi\in C_0^{\infty}(\R)$, we have 

\[
 \int_{\R} V(t,x) \varphi(x) \eta_R (x)  dx = - \int_{\R} u(t,x)   \theta(x) \eta_R (x)  dx
-\int_{\R} V(t,x)   \theta(x) \partial_x \eta_R (x)  dx\,,
\]

\noindent where $\theta(x) =\int_{-\infty}^{x}  \varphi(y)   \ dy$. By  definition of the solution $u$, taking as test function $ \theta(x) \eta_R (x)$ we deduce that 

\begin{align}\label{DISTINTSTRTR}
    \int_{\R} & V(t,x) \ \eta_R (x) \varphi(x) dx = - \int_{0}^{t} \!\! \int_{\R}   \partial_x V(s,x)   \,  v(s,x) \eta_R (x) \varphi(x) \ dx ds \nonumber\\[5pt]
    & - \int_{0}^{t} \!\! \int_{\R}   \partial_x V(s,x) \  \eta_R (x) \varphi(x) \ dx \, {\circ}{dB_s} - \int_{0}^{t} \!\! \int_{\R}   \partial_x V(s,x)   \, v(s,x) \partial_x \eta_R (x) \theta(x)  \ dx ds\nonumber\\[5pt]
    &- \int_{0}^{t} \!\! \int_{\R}   \partial_x V(s,x) \  \partial_x \eta_R (x) \theta(x)  \ dx \, {\circ}{dB_s}-\int_{\R} V(t,x)   \theta(x) \partial_x \eta_R (x)  dx.
\end{align}

Since $V\in L^{\infty}([0,T],L^{1}(\R) )$ taking the limit as $R\rightarrow \infty$ we get 

\begin{equation} \label{DISTINTSTRT}
\begin{aligned}
    \int_{\R} V(t,x) \varphi(x) dx = \\[5pt]
	- \int_{0}^{t} \!\! \int_{\R}   \partial_x V(s,x)   \, v(s,x)  \varphi(x) \ dx ds
- \int_{0}^{t} \!\! \int_{\R}   \partial_x V(s,x) \   \varphi(x) \ dx \, {\circ}{dB_s} .
\end{aligned}
\end{equation}

\bigskip

{\it Step 2: Smoothing.}
Let $\{\rho_{\varepsilon}(x)\}_\varepsilon$ be a family of standard symmetric mollifiers. For any $\varepsilon>0$ and $x\in\R^d$ we use $\rho_\varepsilon(x-\cdot)$ as test function and we  obtain 
$$
\begin{aligned}
      \int_{\R} V(t,y) \rho_\varepsilon(x-y) \, dy  = &\, - \int_{0}^{t}  \int_{\R} \big( v(s,y)  \partial_y V(s,y)  \big)  \rho_\varepsilon(x-y) \ dy ds
		\\[5pt]
    & -  \int_{0}^{t} \!\! \int_{\R} \partial_y V(s,y) \, \rho_\varepsilon(x-y)  \, dy \circ dB_s
\end{aligned}
$$

We put   $V_\varepsilon(t,x)= (V\ast \rho_\varepsilon)(x)$, $v_\varepsilon(t,x)= (v \ast \rho_\varepsilon)(t,x)$ and
$(vV)_\varepsilon(t,x)= (v.V\ast \rho_\varepsilon)(x)$. Then have

\[
    V_{\varepsilon}(t,x) + \int_{0}^{t} v\varepsilon(s,x)  \partial_x V_{\varepsilon}(s,x) \,  ds   +  \int_{0}^{t}   \partial_{x}  V_{\varepsilon}(s,x) \, \circ dB_s
	\]
\[			
  =\int_{0}^{t} \big(\mathcal{R}_{\epsilon}(V,v) \big) (x,s) \,  ds ,
	\]

\noindent where 
$ \mathcal{R}_{\epsilon}(V,v)  = v_\varepsilon \ \partial_x V_\varepsilon  -  (v\partial_x V)_\varepsilon  $.

{\it Step 3: Method of Characteristic.}

We consider the stochastic flow

\[
d X_t^{\epsilon} = v^\varepsilon (t,X_t^{\epsilon}) \, dt + d B_t \, ,  \hspace{1cm}   X_0 = x \,.
\]

Using the same arguments that in steps 3-5-6  of the existence proof we have

\begin{equation}\label{est3}
\E |JX_{t-s}^{\epsilon}|^{2}\leq C.
\end{equation}

Applying the It\^o-Wentzell-Kunita formula   to $ V_{\varepsilon}(t,X_{t}^{\epsilon})$
, see Theorem 8.3 of \cite{Ku2}, we have

\[
   V_{\varepsilon}(t,X_{t}^{\epsilon})  = \int_{0}^{t} \big(\mathcal{R}_{\epsilon}(V,v) \big) (X_s^{\epsilon},s)  ds  .
\]

Hence 

\[
   V_{\varepsilon}(t,x)  =\int_{0}^{t} \big(\mathcal{R}_{\epsilon}(V,v) \big) (X_{t-s}^{-1,\epsilon},s)  ds  .
\]

Multiplying by the test functions $\varphi$ and integrating in $\R$ we obtain

\begin{equation}
   \int V_{\varepsilon}(t,x) \ \varphi(x) dx  =
	\int_{0}^{t}  \int \big(\mathcal{R}_{\epsilon}(V,v) \big) (X_{t-s}^{-1,\epsilon},s) \  \  \varphi(x) \ \, dx \   ds  .
\end{equation}

 \noindent Doing the change of variable we obtain

\begin{equation}
	\int_{0}^{t}  \int \big(\mathcal{R}_{\epsilon}(V,v) \big) (X_{t-s}^{-1,\epsilon},s) \ \varphi(x) \ \, dx \   ds =
	\int_{0}^{t}  \int \big(\mathcal{R}_{\epsilon}(V,v) \big) (x,s) \   JX_{t-s}^{\epsilon}  \varphi(X_{t-s}^{\epsilon}) \ \, dx \   ds .
\end{equation}

\bigskip

{\it Step 4: Convergence of the commutator.}   Now, we observe that $\mathcal{R}_{\epsilon}(V,b)$ converge to zero in
$L^{2}([0,T]\times \R )$. In fact, we have

\[
  (v \ \partial_x V)_{\varepsilon} \rightarrow v \ \partial_x V \ in \ L^{2}([0,T]\times \R ),
\]

and by the dominated convergence theorem we obtain

\[
v_{\epsilon}   \partial_x V_{\varepsilon} \rightarrow v \ \partial_x V \ in \ L^{2}([0,T]\times \R ).
\]

\bigskip

{\it Step 5: Conclusion.}  From step 3  we have 

\begin{equation}\label{conv}
   \int V_{\varepsilon}(t,x) \ \varphi(x) dx  =
		\int_{0}^{t}  \int \big(\mathcal{R}_{\epsilon}(V,v) \big) (x,s) \   JX_{t-s}^{\epsilon}  \varphi(X_{t-s}^{\epsilon}) \ \, dx \   ds ,
\end{equation}

 Using  H\"older's inequality we obtain 

\[
	\E \bigg|\int_{0}^{t}  \int \bigg(\mathcal{R}_{\epsilon}(V,v) \bigg) (x,s) \   JX_{t-s}^{\epsilon}  \varphi(X_{t-s}^{\epsilon}) \ \, dx \   ds \bigg|
	\]
	
	\[
	\leq   	 \bigg(\E \int_{0}^{t}  \int |\big(\mathcal{R}_{\epsilon}(V,v) \big) (x,s)|^{2} \  \, dx \   ds \bigg)^{\frac{1}{2}}
  \bigg(\E \int_{0}^{t}  \int  | JX_{t-s}^{\epsilon} \varphi(X_{t-s}^{\epsilon})|^{2} \ \, dx \ ds \bigg)^{\frac{1}{2}}
\]

\noindent  From step 4 we deduce

\[
 \bigg(\E \int_{0}^{t}  \int |\big(\mathcal{R}_{\epsilon}(V,v) \big) (x,s)|^{2} \  \, dx \   ds \bigg)^{\frac{1}{2}}\rightarrow 0.
\]

From estimation  (\ref{est3}) we obtain 

\[
\bigg(\E \int_{0}^{t}  \int  | JX_{t-s}^{\epsilon} \varphi(X_{t-s}^{\epsilon})|^{2} \ \, dx \ ds \bigg)^{\frac{1}{2}}
\]
\[
\leq C \bigg(\int_{0}^{t}  \int_{\R}  |\varphi(x)|^{2} \ \, dx \ ds \bigg)^{\frac{1}{2}}\leq C \ \int_{\R}   |\varphi(x)|^{2} \ \, dx.
\]

Passing to the limit in equation (\ref{conv}) we conclude that $V=0$. Then we deduce that  $u=0$.

\end{proof}

\section*{Acknowledgements}
    Christian Olivera C. O. is partially supported by  CNPq
through the grant 460713/2014-0 and FAPESP by the grants 2015/04723-2 and 2015/07278-0.


\end{document}